\documentclass[a4paper, bibliography=totoc]{scrartcl}
\usepackage{amsmath}
\usepackage{amssymb}
\usepackage{enumerate}
\usepackage{amsthm}
\usepackage{todonotes}
\usepackage{mathtools} 
\usepackage[UKenglish]{babel}
\usepackage[T1]{fontenc}
\usepackage[utf8]{inputenc}
\usepackage{tikz}
\usepackage{thmtools} 
\usepackage{thm-restate} 
\usepackage{stackrel} 
\usepackage[mathlines]{lineno}

\usepackage[normalem]{ulem}
\usepackage[unicode]{hyperref}

\declaretheorem[name=Theorem,numberwithin=section]{thm} 
\newtheorem*{thm*}{Theorem}

\newtheorem*{define*}{Definition}
\newtheorem{define}[thm]{Definition}

\newtheorem*{lemma*}{Lemma}
\newtheorem{lemma}[define]{Lemma}

\newtheorem*{algorithm*}{Algorithm}

\newtheorem*{construction*}{Construction}

\newtheorem*{prop*}{Proposition}

\newtheorem*{obs*}{Observation}

\newtheorem*{fact*}{Fact}

\newtheorem*{remark*}{Remark}

\newtheorem*{quest*}{Question}
\newtheorem{quest}[define]{Question}

\newtheorem*{cor*}{Corollary}
\newtheorem{cor}[define]{Corollary}

\newtheorem*{conjecture*}{Conjecture}

\newtheorem*{question*}{Question}

\newtheorem*{example*}{Example}

\numberwithin{claimcounter}{define}
\newtheorem*{claim*}{Claim}

\newcommand{\hl}[1]{#1}
\newcommand{\mf}{\mathfrak}
\newcommand{\Sph}{\mathbb{S}}

\newcommand{\R}{\mathbb{R}}
\newcommand{\Q}{\mathbb{Q}}
\newcommand{\Z}{\mathbb{Z}}
\newcommand{\lip}{\operatorname{Lip}}

\newcommand{\dist}{\operatorname{dist}}
\newcommand{\N}{\mathbb{N}}

\newcommand{\id}{\operatorname{id}}

\DeclareMathOperator{\diam}{diam}

\newcommand{\mb}[1]{\mathbf{#1}}

\newcommand{\mc}[1]{\mathcal{#1}}

\newcommand{\abs}[1]{\left|#1\right|}
\newcommand{\lnorm}[2]{\left\|#2\right\|_#1}

\newcommand{\supnorm}[1]{\lnorm{\infty}{#1}}
\newcommand{\norm}[1]{\left\|#1\right\|}

\newcommand{\set}[1]{\left\{#1\right\}}
\newcommand{\cl}[1]{\overline{#1}}

\def\XXint#1#2#3{{\setbox0=\hbox{$#1{#2#3}{\int}$ }
		\vcenter{\hbox{$#2#3$ }}\kern-.6\wd0}}

\title{Porosity phenomena of non-expansive, Banach space mappings.}
\author{Michael Dymond}
\begin{document}
	\maketitle
	\abstract{For any non-trivial convex and bounded subset $C$ of a Banach space, we show that outside of a $\sigma$-porous subset of the space of non-expansive mappings $C\to C$, all mappings have the maximal Lipschitz constant one witnessed locally at typical points of $C$. This extends a result of Bargetz and the author from separable Banach spaces to all Banach spaces and the proof given is completely independent. We further establish a fine relationship between the classes of exceptional sets involved in this statement, captured by the hierarchy of notions of $\phi$-porosity.}
	\section{Introduction}	
	The set of strict contractions $C\to C$, where $C$ is a closed, convex and bounded set, is of particular interest in fixed point theory. Recall that Banach's classical fixed point theorem states that every strict contraction of a complete metric space has a unique fixed point, to which all iterations converge. However, if we wish to view the strict contractions as elements of a complete metric space, we arrive at the space of non-expansive mappings $C\to C$ (or mappings $C\to C$ with Lipschitz constant at most one)  and the conclusion of Banach's fixed point theorem need not be valid for such mappings. Thus it is a natural question to ask how many non-expansive mappings $C\to C$ possess a unique fixed point. The size of the set of strict contractions in spaces of non-expansive mappings is clearly relevant to this question. 
	
	When $C$ is a subset of a Hilbert space, De Blasi and Myjak~\cite{DBM1989porosite} prove that the set of strict contractions is a negligible subset of the space of non-expansive mappings. More precisely, they show that the strict contractions form a $\sigma$-porous subset. The argument of \cite{DBM1989porosite} relies heavily on Kirszbraun's theorem for extending Lipschitz mappings between Hilbert spaces. Thus, these methods do not transfer to the general Banach space setting, leading Reich~\cite{reich_genericity} to pose the question of whether this porosity result holds in Banach spaces.
	\begin{quest}[Reich~\cite{reich_genericity}]\label{q:reich}
		Let $X$ be a Banach space, $C\subseteq X$ be a closed, bounded, convex, non-singleton, non-empty set and $\mc{M}=\mc{M}(C)$ denote the space of non-expansive mappings $C\to C$ equipped with the supremum metric. Is the set
		\begin{equation*}
			\set{f\in \mc{M}\colon \lip(f)<1}
		\end{equation*}
		of strict contractions $C\to C$ a $\sigma$-porous subset of $\mc{M}$?
	\end{quest}
	Bargetz and the author answer Reich's question affirmatively in \cite{bargetz_dymond2016}. The present note, provides an independent answer to Reich's question and further strengthens the results of \cite{bargetz_dymond2016}.
	
	Despite strict contractions only forming a very small subset of the space of non-expansive mappings, it turns out that most non-expansive mappings have a unique fixed point. In the Banach space setting, Reich and Zavlaski~ show that outside of a $\sigma$-porous subset of the space of non-expansive mappings $C\to C$ all mappings are contractive in the sense of Rakotch~\cite{rakotch1962note}, meaning they decrease distances between points by a factor which is allowed to behave as a monotonic function of the distance between the two points; for a more precise formulation see for example \cite[Definition~3.3.15]{istruactescu2001fixed}. Such Rakotch contractive mappings are shown to have the fixed point property in \cite{rakotch1962note}.
	
	The problems discussed above have also been studied in more general metric spaces and for more general classes of mappings; see, for example, \cite{RZ2016two}, \cite{BDR2017porosity} and \cite{BDMR2021existence}. 
	
	When $C$ is a subset of a separable Banach spaces, the paper \cite{bargetz_dymond2016} proves a stronger result than just the $\sigma$-porosity of the set of strict contractions. In fact \cite{bargetz_dymond2016} proves that outside of a $\sigma$-porous subset of the space of non-expansive mappings $C\to C$, all mappings $f$ have the property that the set $R(f)$ defined below in Theorems~\ref{thm:main} and \ref{thm:dual} is a residual subset of $C$. The quantity $\lip(f,x)$ may be informally described as the local Lipschitz constant of $f$ at $x$ and is defined in the next section by \eqref{eq:lipfx}. Thus, roughly speaking, \cite{bargetz_dymond2016} proves that outside of a $\sigma$-porous subset of the space of non-expansive mappings, all mappings have the maximal Lipschitz constant one and this maximal Lipschitz constant is witnessed locally at typical points of the domain $C$. 
	
	One contribution of the present note is to extend this result to all Banach spaces. We prove the following theorem, which is stronger than both main results (Theorems~2.1 and 2.2) of \cite{bargetz_dymond2016}; in particular, it removes the separability condition on the Banach space from \cite[Theorem~2.2]{bargetz_dymond2016}.
	\begin{thm}\label{thm:main}
		Let $X$ be a Banach space, $C\subseteq X$ be a closed, bounded, convex, non-singleton, non-empty set and $\mc{M}=\mc{M}(C)$ denote the space of non-expansive mappings $C\to C$ equipped with the supremum metric. Then there is a $\sigma$-lower porous subset $\mc{N}$ of $\mc{M}$ such that for every $f\in \mc{M}\setminus \mc{N}$ the set
		\begin{equation*}
			R(f):=\set{x\in C\colon \lip(f,x)=1}
		\end{equation*}
		is a residual subset of $C$.
	\end{thm}
	The notion of lower porosity appearing in Theorem~\ref{thm:main} is the same as the notion of \cite[Theorems~2.1 and 2.2]{bargetz_dymond2016}, where it is simply called porosity. The present note uses an extended terminology because we also consider other notions of porosity, in particular upper porosity.
	
	Theorem~\ref{thm:main} asserts that after excluding a negligible subset $\mc{N}$ of the space $\mc{M}$ of non-expansive mappings $C\to C$, all remaining mappings $f\in\mc{M}\setminus\mc{N}$ have the property that the set $C\setminus R(f)$ is a negligible subset of $C$. Thus, this statement concerns two types of negligible sets, the former in the space $\mc{M}$ and the latter in $C$. The next theorem reveals that these two classes of negligible sets are intrinsically linked. More precisely, weakening the sense of negligibility of the set $\mc{N}$, and thus excluding a larger, but still negligible subset of $\mc{M}$, achieves a stronger sense of negligibilty of the sets $C\setminus R(f)$ for the remaining mappings $f\in\mc{M}\setminus \mc{N}$. The next theorem further determines a precise relationship between the degree of negligibility of sets $C\setminus R(f)$ for $f\in\mc{M}\setminus\mc{N}$ and the degree of negligibility of $\mc{N}$ as a subset of $\mc{M}$. This dependency is expressed in the form of a pair of gauge functions $(\phi,\hl{\xi})$.
	\begin{thm}\label{thm:dual}
		Let $X$ be a Banach space, $C\subseteq X$ be a closed, bounded, convex, non-singleton, non-empty set and $\mc{M}=\mc{M}(C)$ denote the space of non-expansive mappings $C\to C$ equipped with the supremum metric. Let $K>1$ and $\phi,\hl{\xi}\colon(0,1/K)\to(0,\infty)$ be strictly increasing concave functions satisfying
		\begin{equation*}
			\phi(t)\hl{\xi}(t)\geq \frac{t}{K}\text{ for all $t\in(0,1/K)$}\quad\text{ and }\quad\lim_{t\to 0}\hl{\xi}(t)=0.
		\end{equation*}
		Then there is a $\sigma$-$\hl{\xi}$-lower porous subset $\mc{N}$ of $\mc{M}$ such that for every mapping $f\in \mc{M}\setminus \mc{N}$ the set
		\begin{equation*}
			R(f):=\set{x\in C\colon \lip(f,x)=1}
		\end{equation*}
		is the complement of a $\sigma$-$\phi$-upper porous subset of $C$.
	\end{thm}
	Full explanations of the different notions of porosity appearing in this paper are given in \hl{S}ubsection~\ref{subsec:porous}. For now it suffices to know that for increasing, concave functions $\hl{\xi}\colon (0,\infty)\to (0,\infty)$ the notions of $\hl{\xi}$-porosity form a hierarchy and that the notions of porosity appearing in Theorem~\ref{thm:dual} are weaker than the standard form of porosity, but stronger than nowhere density (these comparisons need not be strict). A more detailed analysis of Theorem~\ref{thm:dual} is postponed until \hl{S}ubsection~\ref{subsec:thmdual}.
	
	The proofs of Theorems~\ref{thm:main} and \ref{thm:dual} are completely independent of \cite{bargetz_dymond2016}, so the present work may also be of interest as a new approach to answering Reich's question (Question~\ref{q:reich}). It may also be of note that the given answer to Reich's question in \hl{S}ubsection~\ref{subsec:Reich} is entirely elementary in the sense that it does not require knowledge of any advanced theorems. There are no preexisting elementary proofs of the positive answer to Reich's question, even in the restricted Hilbert space setting. Even the proofs of Theorems~\ref{thm:main} and \ref{thm:dual} may be described as almost entirely elementary. The only reason for writing almost entirely instead of entirely in the previous sentence is that in one step the existence of a maximal $s$-separated set is used. This is a consequence of Zorn's lemma.
	\section{Preliminaries and Notation.}
	\subsection{General notation.}
	Given a metric space $(M,d)$, a point $x\in M$ and $r>0$ we let $B_{M}(x,r)$ denote the open ball in $M$ with centre $x$ and radius $r$. For the corresponding closed ball we replace $B_{M}$ with $\cl{B}_{M}$. For $s,\theta>0$ a subset $\Gamma$ of $M$ will be called \emph{$s$-separated} if $d(x,y)\geq s$ for every $x,y\in \Gamma$ and \emph{$\theta$-dense} if $\bigcup_{x\in\Gamma}\cl{B}_{M}(x,\theta)=M$. 
	
	The origin and unit sphere of a normed space $(X,\norm{-})$ will be denoted by $0_{X}$ and $\Sph_{X}$ respectively. Since we only ever work with one fixed norm space, we always write $X$ instead of $(X,\norm{-})$. For a subset $C$ of $X$ we let
	\begin{equation*}
		\diam C:=\sup\set{\norm{y-x}\colon x,y\in C}.
	\end{equation*}
	The quantity $\diam C$ will appear often as an upper bound for $\norm{x}$ for points $x\in C$ when $C$ contains $0_{X}$. Given two points $x,y\in X$ we write $[x,y]$ for the closed line segment with endpoints $x$ and $y$.
	
	We call a function $\phi\colon (0,\infty)\to (0,\infty)$ \emph{increasing} if $\phi(s)\leq \phi(t)$ whenever $s,t\in (0,\infty)$ and $s\leq t$. If the stronger condition holds with the non-strict inequalities replaced by strict inequalities, we refer to $\phi$ as \emph{strictly increasing}. The inverse function of a strictly increasing $\phi$ is denoted by $\phi^{-1}\colon (\inf\phi,\sup\phi)\to (0,\infty)$. For functions $\phi,\hl{\xi}\colon (0,\infty)\to (0,\infty)$ we will make use of the standard big Theta notation with respect to the asymptotic behaviour as $t\to 0$. Thus, we will write $\phi(t)\in \Theta(\hl{\xi}(t))$ to signify the existence of constants $\eta,a,b>0$ such that $a\hl{\xi}(t)\leq \phi(t)\leq b\hl{\xi}(t)$ for all $t\in (0,\eta)$.

	Given a normed space $X$, a subset $C$ of $X$ and a mapping $f\colon C\to X$ we let
	\begin{equation*}
		\lip(f):=\hl{\sup}\set{\frac{\norm{f(y)-f(x)}}{\norm{y-x}}\colon x,y\in C,\,y\neq x}.
	\end{equation*}
	If $\lip(f)<\infty$ then $f$ is called \emph{Lipschitz}, if $\lip(f)\leq 1$ then $f$ is called \emph{non-expansive} and if $\lip(f)<1$ then $f$ is referred to as a \emph{strict contraction}. For $x\in C$ we further distinguish the quantity
	\begin{equation}\label{eq:lipfx}
		\lip(f,x):=\lim_{r\to 0}\sup\set{\frac{\norm{f(y)-f(x)}}{\norm{y-x}}\colon y\in C,\, 0<\norm{y-x}\leq r},
	\end{equation} 
	which may be thought of as the Lipschitz constant of $f$ witnessed locally at the point $x$. Further, for $r\in (0,\infty)$ we let
	\begin{equation*}
		\lip(f,x,r):=\sup\set{\frac{\norm{f(y)-f(x)}}{\norm{y-x}}\colon y\in C,\, 0< \norm{y-x}\leq r}.
	\end{equation*}
	The quantity $\lip(f,x,r)$ can be described as the Lipschitz constant of $f$ witnessed at the point $x$ at scale $r$. Note that 
	\begin{equation*}
		\lim_{r\to 0}\lip(f,x,r)=\lip(f,x).
	\end{equation*}
	The sets $R(f)$ given in Theorems~\ref{thm:main} and \ref{thm:dual} are defined in terms of the quantity $\lip(f,x)$. We record here a basic property of these sets.
	\begin{lemma}\label{lemma:Gdelta}
		Let $X$ be a normed space, $C\subseteq X$ and $f\colon C\to C$ be a non-expansive mapping. Then the set
		\begin{equation*}
			R(f):=\set{x\in C\colon \lip(f,x)=1}
		\end{equation*}
		is a relatively $G_{\delta}$ subset of $C$.
	\end{lemma}
	\begin{proof}
		It suffices to observe that the set $R(f)$ may be written as
		\begin{equation*}
			\bigcap_{\lambda\in \Q\cap (0,1)}\bigcap_{n\in\N}\set{x\in C\colon \exists y\in C\cap B_{X}(x,1/n)\text{ s.t.}\norm{f(y)-f(x)}>\lambda\norm{y-x}}
		\end{equation*} 
		and that each of the sets participating in the above intersection is relatively open in $C$.
	\end{proof}
	\subsection{Spaces of non-expansive mappings.}\label{subsect:nonexp}
	Given a normed space $X$ and a bounded and convex subset $C$ of $X$ we consider the space $\mc{M}(C)$ of non-expansive mappings $C\to C$, defined by
	\begin{equation*}
		\mc{M}(C):=\set{f\colon C\to C\colon \lip(f)\leq 1}.
	\end{equation*}
	The set $\mc{M}(C)$ will be equipped with the supremum metric
	\begin{equation*}
		d_{\infty}(f,g)=\supnorm{g-f},\qquad f,g\in\mc{M}(C).
	\end{equation*}
	If $X$ is a Banach space and $C\subseteq X$ is closed then the pair $(\mc{M}(C),d_{\infty})$ is a complete metric space. In the remainder of the paper we shorten the notation $(\mc{M}(C),d_{\infty})$ to $\mc{M}(C)$ and favour writing $\supnorm{g-f}$ instead of $d_{\infty}(f,g)$.
	\subsection{Porosity and $\sigma$-porosity.}\label{subsec:porous}
	The notions of $\phi$-upper and $\phi$-lower porosity are defined according to \cite[Definition~2.1]{Zajicek05} and \cite[2.7]{zajivcek1976sets}.

	Let $(M,d)$ be a metric space, $P\subseteq M$, $x\in M$, $\eta>0$ and $\phi\colon (0,\eta)\to (0,\infty)$ be an increasing function. For each $r>0$ we consider the quantity
	\begin{equation*}
		\gamma(x,r,P):=\sup\set{s>0 \colon \exists x'\in M\text{ such that }B_{M}(x',s)\subseteq B_{M}(x,r)\setminus P},
	\end{equation*}
	where we interpret the supremum of the empty set as $-\infty$. We say that $P$ is $\phi$-upper porous at the point $x$ if 
	\begin{equation*}
		\limsup_{r\to 0}\frac{\phi(\gamma(x,r,P))}{r}>0.
	\end{equation*}
	We say that $P$ is $\phi$-lower porous at the point $x$ if
	\begin{equation*}
		\liminf_{r\to 0}\frac{\phi(\gamma(x,r,P))}{r}>0.
	\end{equation*}
	If $\gamma(x,r,P)=-\infty$ then the quantity $\phi(\gamma(x,r,P))$ is not defined. Therefore, the conditions above implicitly include the condition $\gamma(x,r,P)>0$ for all $r>0$.
	
	The set $P$ is said to be $\phi$-upper porous if it is $\phi$-upper porous at every point $x\in P$. $\phi$-lower porous sets are defined analogously. 
	
	Note that $\phi$-lower porosity is stronger than $\phi$-upper porosity for every $\phi$. Moreover, the notions of $\phi$-porosity form a hierarchy ordered according to the asymptotic behaviour of $\phi(t)$ as $t\to 0+$. Indeed the notions of $\phi_{1}$-porosity are stronger than the corresponding notions of $\phi_{2}$-porosity whenever $\limsup_{t\to 0}\frac{\phi_{1}(t)}{\phi_{2}(t)}<\infty$. Accordingly, the notions of $\phi_{1}$- and $\phi_{2}$-porosity coincide whenever $\phi_{1}(t)\in \Theta(\phi_{2}(t))$. 
	
	For functions $\phi\in \Theta(1)$, or even those with just $\liminf_{t\to 0}\phi(t)>0$, the notions of $\phi$-upper and $\phi$-lower porosity coincide and equal the standard notion of nowhere density. For functions $\phi\in\Theta(t)$ the notions of $\phi$-upper and $\phi$-lower porosity are precisely the standard notions of upper and lower porosity. When $\phi\in \Theta(t)$ we will write upper porous and lower porous instead of $\phi$-upper porous and $\phi$-lower porous respectively.  
	
	Thus, the hierarchy of notions of $\phi$-upper and $\phi$-lower\hl{-}porosity has its weakest notion, namely nowhere density, at one end and its strongest notion, namely the standard notions of porosity, at the other end. In between we have intermediate notions, corresponding to functions $\phi(t)$ lying asymptotically in between the constant function and the identity for $t\to 0$. Notable examples of such functions are $\phi_{p}(t)=t^{p}$ for $p\in (0,1)$. In fact the notions of $\phi_{p}$-porosity receive special attention in \cite{zajivcek1976sets}.
	
	A stronger property than $\phi$-upper or $\phi$-lower porosity is given when the porosity condition of $P$ is satisfied not just at all points $x\in P$ but at every point $x$ in the whole space $M$. For example the set $P=\set{\frac{1}{n}\colon n\in\Z\setminus\set{0}}$ is porous in $\R$, but it is not porous at the point $0\in\R\setminus P$. We highlight this because all of the sets shown to be porous in the present work will actually be shown to possess the stronger property of being porous at every point of the space.
	
	There are differing conventions in the literature concerning whether upper or lower porosity is taken as the standard form of porosity. For example in \cite{DBM1989porosite}, \cite{reich2001set} and \cite{bargetz_dymond2016} porous means lower porous, whilst in \cite{zajivcek1976sets}, \cite{Zajicek05} and \cite{lindenstrauss2011frechet}  porous means upper porous.

	In the remainder of the paper we will avoid using the $\gamma$ notation in the definitions of porosity above. Instead, we will make use of the following equivalent formulations of the notions of $\phi$-porosity for strictly increasing, concave $\phi$ satisfying $\lim_{t\to 0}\phi(t)=0$. \hl{The characterisations given in Lemma~\ref{lemma:def_porous} can be seen as anologues of classical characterisations of upper and lower porosity.}
	\begin{lemma}\label{lemma:def_porous}
		Let $(M,d)$ be a perfect metric space, $P\subseteq M$, $q\in M$, $\eta>0$ and $\phi\colon (0,\eta)\to (0,\infty)$ be a strictly increasing, concave function with $\lim_{t\to 0}\phi(t)=0$. Then the following statements hold:
		\begin{enumerate}[(a)]
			\item\label{def_upper_por} $P$ is $\phi$-upper porous at $q$ if and only if there exists $\alpha\in (0,1)$ such that for every $\varepsilon>0$ there exists $q'\in M$ such that $0<d(q,q')\leq \varepsilon$ and $B_{M}(q',\phi^{-1}(\alpha d(q,q')))\cap P=\emptyset$.
			\item\label{def_lower_por}$P$ is $\phi$-lower porous at $q$ if and only if there exists $\varepsilon_{0}>0$ and $\beta\in (0,1)$ such that for every $\varepsilon\in (0,\varepsilon_{0})$ there exists $q'\in M$ such that $d(q,q')\leq \varepsilon$ and $B_{M}(q',\phi^{-1}(\beta\varepsilon))\cap P=\emptyset$.
		\end{enumerate}
	\end{lemma}
	\begin{proof}
		From the concavity of $\phi$ and $\lim_{t\to 0}\phi(t)=0$ we may conclude that
		\begin{equation*}
			K=K(\phi):=\sup\set{\frac{t}{\phi(t)}\colon t\in (0,\eta)}<\infty.
		\end{equation*}
		Moreover, since $\phi\colon (0,\eta)\to (0,\infty)$ is strictly increasing with $\lim_{t\to 0}\phi(t)=0$, we have that the inverse function $\phi^{-1}$ exists and is defined on the interval $(0,\sup \phi)$.
		
		We begin by proving \eqref{def_upper_por}. Let $\alpha\in (0,1)$ be given by the condition of \eqref{def_upper_por} and fix $\varepsilon>0$. Then there exists $q'\in M$ with $0<d(q,q')< \min\set{\frac{\varepsilon}{1+K},\frac{\sup\phi}{\alpha}}$ and $B_{M}(q',\phi^{-1}(\alpha d(q,q')))\cap P=\emptyset$. Note that $\phi^{-1}(\alpha d(q,q'))\leq K\alpha d(q,q')\leq Kd(q,q')$ and so we derive
		\begin{equation*}
			B_{M}(q',\phi^{-1}(\alpha d(q,q')))\subseteq B_{M}(q,(1+K)d(q,q'))\setminus P.
		\end{equation*}
		For $r:=(1+K)d(q,q')$ we now have $\gamma(q,r,P)\geq \phi^{-1}(\alpha d(q,q'))$, $r\in(0,\varepsilon)$ and
		\begin{equation*}
			\frac{\phi(\gamma(q,r,P))}{r}\geq \frac{\alpha d(q,q')}{(1+K)d(q,q')}=\frac{\alpha}{1+K}.
		\end{equation*}
		This proves $\limsup_{r\to 0}\frac{\phi(\gamma(q,r,P))}{r}\geq \frac{\alpha}{1+K}$ and thus the $\phi$-upper porosity of $P$ at $q$. 
		
		Conversely, suppose that $P$ is $\phi$-upper porous at $q$ and set 
		\begin{equation*}
			\alpha':=\limsup_{r\to 0}\frac{\phi(\gamma(q,r,P))}{r}>0 \quad\text{ and }\quad \alpha:=\frac{1}{4}\min\set{\alpha',1}. 
		\end{equation*}
		Fix $\varepsilon>0$. Then we may choose $r\in (0,\varepsilon)$ so that $\frac{\phi(\gamma(q,r,P))}{r}> 2\alpha$, or equivalently, $\gamma(q,r,P)>\phi^{-1}(2\alpha r)$. It follows that there exists $q''\in M$ such that
		\begin{equation*}
			B_{M}(q'',\phi^{-1}(2\alpha r))\subseteq B_{M}(q,r)\setminus P.
		\end{equation*}
		Due to the perfectness of $M$, we may choose $q'\in B_{M}(q'',\phi^{-1}(\alpha r))\setminus\set{q}$. Moreover, from the convexity of $\phi^{-1}$ and $\lim_{t\to 0}\phi^{-1}(t)=0$, we get $\phi^{-1}(\alpha r)\leq \frac{1}{2}\phi^{-1}(2\alpha r)$. Hence, $0<d(q,q')<r<\varepsilon$ and $B_{M}(q',\phi^{-1}(\alpha r))\cap P=\emptyset$. This verifies the $(\alpha,\varepsilon)$-condition of \eqref{def_upper_por}. 
		
		We turn now to \eqref{def_lower_por}. Assume first that there exist $\varepsilon_{0}>0$ and $\beta\in (0,1)$ as in the condition of \eqref{def_lower_por}. We may assume that $\varepsilon_{0}<\sup\phi/\beta$. Let $\varepsilon\in (0,\varepsilon_{0})$ and then choose $q'\in M$ according to the condition of \eqref{def_lower_por}. Since $\phi^{-1}(\beta \varepsilon)\leq K\beta\varepsilon\leq K\varepsilon$, we deduce that 
		\begin{equation*}
			B_{M}(q',\phi^{-1}(\beta\varepsilon))\subseteq B_{M}(q,(1+K)\varepsilon)\setminus P.
		\end{equation*}
		Therefore $\gamma(q,(1+K)\varepsilon,P)\geq \phi^{-1}(\beta\varepsilon)$ and
		\begin{equation*}
			\frac{\phi(\gamma(q,(1+K)\varepsilon,P))}{(1+K)\varepsilon}\geq \frac{\beta \varepsilon}{(1+K)\varepsilon}=\frac{\beta}{1+K}.
		\end{equation*}
		Since we have verified the above inequality for an arbitrary $\varepsilon\in(0,\varepsilon_{0})$, we conclude that $\liminf_{r\to 0}\frac{\phi(\gamma(q,r,P))}{r}\geq \frac{\beta}{1+K}>0$. Hence, $P$ is $\phi$-lower porous at $q$.
		
		Conversely, suppose that $P$ is $\phi$-lower porous at $q$ and set
		\begin{equation*}
			\beta':=\liminf_{r\to 0+}\frac{\phi(\gamma(q,r,P))}{r}>0\quad\text{ and }\quad \beta:=\frac{1}{2}\min\set{\beta',1}.
		\end{equation*}
		Now we may choose $\varepsilon_{0}>0$ sufficiently small so that 
		\begin{equation*}
			\frac{\phi(\gamma(q,r,P))}{r}>\beta \qquad \text{ for all }r\in (0,\varepsilon_{0}).
		\end{equation*}
		Let $\varepsilon\in (0,\varepsilon_{0})$. Then $\frac{\phi(\gamma(q,\varepsilon,P))}{\varepsilon}>\beta$, or equivalently, $\gamma(q,\varepsilon,P)>\phi^{-1}(\beta\varepsilon)$, so we may find $q'\in M$ with 
		\begin{equation*}
			B_{M}(q',\phi^{-1}(\beta\varepsilon))\subseteq B_{M}(q,\varepsilon)\setminus P.
		\end{equation*}
		Now we have $d(q,q')\leq \varepsilon$ and $B_{M}(q',\phi^{-1}(\beta\varepsilon))\cap P=\emptyset$. This verifies the $(\varepsilon_{0},\beta)$-condition of \eqref{def_lower_por}.
	\end{proof}

	\subsection{Discussion of Theorem~\ref{thm:dual}}\label{subsec:thmdual}
	Theorem~\ref{thm:dual} actually contains Theorem~\ref{thm:main}. Indeed, for the choice $\hl{\xi}(t)=\frac{t}{1+t^{p}}$ and $\phi(t)=1+t^{p}$ with $p\in (0,1)$, we have that $\hl{\xi}(t) \in \Theta(t)$ and $\phi(t)\in\Theta(1)$. Therefore, $\hl{\xi}$- and $\phi$-porosity become simply porosity and nowhere density respectively, so that Theorem~\ref{thm:dual} with these choices of $\hl{\xi}$ and $\phi$ is precisely Theorem~\ref{thm:main}. The proof of Theorem~\ref{thm:dual} distinguishes the special case of Theorem~\ref{thm:main}. Therefore, we provide a proof of Theorem~\ref{thm:main} separately to that of Theorem~\ref{thm:dual}.
	
	Due to the antisymmetric conditions on $\phi$ and $\hl{\xi}$ in Theorem~\ref{thm:dual}, there is no admissible choice of $\phi$ and $\hl{\xi}$ which delivers the natural dual of Theorem~\ref{thm:main}: we are not able to assert that the sets $C\setminus R(f)$ are $\sigma$-upper porous for all $f\in\mc{M}$ outside of a nowhere dense set. In contrast, for any weaker form of $\phi$-upper porosity there is an admissible choice of $\hl{\xi}$ so that the sets $C\setminus R(f)$ in the conclusion of Theorem~\ref{thm:dual} are $\sigma$-$\phi$-upper porous. This is established by Lemma~\ref{lemma:pairs} and Corollary~\ref{cor:all_weaker_phi} below. Theorem~\ref{thm:dual} therefore invites the following open question:
	\begin{quest}\label{quest:dual}[Is the dual statement to Theorem~\ref{thm:main} valid?]
		Let $X$ be a Banach space, $C\subseteq X$ be a closed, bounded, convex, non-singleton, non-empty set and $\mc{M}=\mc{M}(C)$ denote the space of non-expansive mappings $C\to C$ equipped with the supremum metric. For $f\in\mc{M}$, let 
		\begin{equation*}
			R(f):=\set{x\in C\colon \lip(f,x)=1}.
		\end{equation*} Does there exist a residual subset $\mc{G}$ of $\mc{M}$ so that for all $f\in \mc{G}$ the set $C\setminus R(f)$ is $\sigma$-upper porous?
	\end{quest}
	\begin{lemma}\label{lemma:pairs}
		Let $\eta\in(0,1)$ and $\phi\colon (0,\eta)\to (0,\infty)$ be a strictly increasing, concave function with 
		\begin{equation*}
			\qquad\lim_{t\to 0}\frac{\phi(t)}{t}=\infty.
		\end{equation*}
		Then there exist $K>1/\eta$ and a strictly increasing, concave function $\hl{\xi}\colon (0,1/K)\to (0,\infty)$ such that 
		\begin{equation*}
			\frac{t}{K}\leq \phi(t)\hl{\xi}(t)\leq Kt \quad\text{ for all }t\in (0,1/K),\quad\text{ and }\quad \lim_{t\to 0}\hl{\xi}(t)=0.
		\end{equation*}	 
	\end{lemma}
	\begin{proof}
		If $\inf\phi=\lim_{t\to 0}\phi(t)>0$ it suffices to choose $K>\max\set{\frac{1}{\eta},\frac{1}{\inf \phi},\sup\phi}$ arbitrarily and take $\hl{\xi}(t)=t$. Therefore, we may assume that $\lim_{t\to 0}\phi(t)=0$. 
		
		Choose $t_{0}\in (0,\eta)$ such that $\phi$ is differentiable at $t_{0}$. Then we may define a strictly increasing, concave function $\hat{\phi}\colon (0,\infty)\to (0,\infty)$ by
		\begin{equation*}
			\hat{\phi}(t)=\begin{cases}
				\phi(t) & \text{ if }t\in (0,t_{0}),\\
				\phi(t_{0})+\phi'(t_{0})(t-t_{0}) & \text{ if }t\in[t_{0},\infty).
			\end{cases}
		\end{equation*}
		Note that the function $\frac{t}{\hat{\phi}(t)}$ is increasing, because $\hat{\phi}$ is concave and satisfies $\lim_{t\to 0}\hat{\phi}(t)=0$. It follows that
		\begin{equation*}
			\frac{s}{\hat{\phi}(s)}\leq \max\set{1,\frac{s}{t}}\frac{t}{\hat{\phi}(t)}
		\end{equation*}
		for all $s,t\in (0,\infty)$. Let $\hl{\xi}\colon (0,\infty)\to (0,\infty)$ denote the least concave majorant of the function $\frac{t}{\hat{\phi}(t)}$. Then, by \cite[Lemma~1]{peetre1970concave}, there is $L>0$ such that 
		\begin{equation*}
			\frac{t}{L\hat{\phi}(t)}\leq \hl{\xi}(t)\leq \frac{Lt}{\hat{\phi}(t)}
		\end{equation*}
		for all $t\in (0,\infty)$, so that $\lim_{t\to 0}\hl{\xi}(t)=0$. It remains to choose $K>\max\set{L,1/t_{0}}$ sufficiently large so that $\hl{\xi}'(t)>0$ for all $t\in (0,1/K)$.
	\end{proof}
	The next corollary of Theorem~\ref{thm:dual} and Lemma~\ref{lemma:pairs} highlights the pertinence of Question~\ref{quest:dual}. The only gauge functions $\phi$ not captured by it are those for which $\phi$-upper porosity is precisely upper porosity. 
	\begin{cor}\label{cor:all_weaker_phi}
		Let $X$ be a Banach space, $C\subseteq X$ be a closed, bounded, convex, non-singleton, non-empty set and $\mc{M}=\mc{M}(C)$ denote the space of non-expansive mappings $C\to C$ equipped with the supremum metric. For $f\in\mc{M}$, let 
		\begin{equation*}
			R(f):=\set{x\in C\colon \lip(f,x)=1}.
		\end{equation*}
		Let $\eta\in (0,1)$ and let $\phi\colon (0,\eta)\to (0,\infty)$ be a strictly increasing, concave function with 
		\begin{equation*}
			\lim_{t\to 0}\frac{\phi(t)}{t}=\infty.
		\end{equation*}
		Then there is a residual subset $\mc{G}$ of $\mc{M}$ such that $C\setminus R(f)$ is $\sigma$-$\phi$-upper porous for every $f\in \mc{G}$. 
	\end{cor}
	\begin{proof}
		Let $K>1$ and $\hl{\xi}$ be given by the conclusion of Lemma~\ref{lemma:pairs}. Then apply Theorem~\ref{thm:dual} to the pair of functions $\phi,\hl{\xi}\colon (0,1/K)\to (0,\infty)$.
	\end{proof}

	\section{Perturbation of Banach space Lipschitz mappings.}\label{sec:perturb}
	The present section details some methods of modifying a given non-expansive mapping of a normed space, so that its Lipschitz constant on targeted sets increases. These methods will be employed in the next section to establish the main porosity results.
	\begin{lemma}\label{lemma:flat}
		Let $X$ be a normed space, $C\subseteq X$ be a convex set containing $0_{X}$ and $0<\delta<r$. Then there exists a mapping $\Phi\colon C\to C$ such that
		\begin{enumerate}[(i)]
			\item\label{Phi1} $\Phi(x)=0_{X}$ for all $x\in C\cap \cl{B}_{X}(0_{X},\delta)$.
			\item\label{Phi2} $\Phi(x)=x$ for all $x\in C\setminus B_{X}(0_{X},r)$.
			\item\label{Phi3} $\lip(\Phi)\leq 1+\frac{\delta}{r-\delta}$.
			\item\label{Phi4} $\supnorm{\Phi-\id_{C}}\leq \delta$.
		\end{enumerate}
	\end{lemma}
	\hl{\begin{remark*}
			The mapping $\Phi$ is similar to the mapping $R_{\varepsilon}$ considered independently in \cite[Lemma~4.2]{medjic2022successive}.
	\end{remark*}}
	\begin{proof}
		We define $\Phi\colon C\to C$ by
		\begin{equation*}
			\Phi(x)=\begin{cases}
				0_{X} & \text{ if }x\in C\cap \cl{B}_{X}(0_{X},\delta),\\
				\left(\frac{r-\frac{r\delta}{\norm{x}}}{r-\delta}\right)x & \text{ if }x\in C\cap B_{X}(0_{X},r)\setminus \cl{B}_{X}(0_{X},\delta),\\
				x & \text{ if } x\in C\setminus B_{X}(0_{X},r).
			\end{cases}
		\end{equation*}
		Observe that $\Phi(x)$ is always a convex combination of $x\in C$ and $0_{X}\in C$. Thus, it is clear that $\Phi$ is a well-defined mapping $C\to C$.
		We verify the conditions \eqref{Phi1}--\eqref{Phi4} for $\Phi$. Conditions \eqref{Phi1} and \eqref{Phi2} are already apparent. For \eqref{Phi3} we note first that $\Phi$ is clearly continuous. Moreover, any line segment in $C$ is partitioned into at most $5$ line segments by the subsets of $C$ distinguished in the above formula for $\Phi$. It therefore suffices to check that the Lipschitz constant of $\Phi$ restricted to each of these subsets is at most $1+\frac{\delta}{r-\delta}$. The only subset for which this is not immediately clear is $C\cap B_{X}(0_{X},r)\setminus \cl{B}_{X}(0_{X},\delta)$. Thus, we consider arbitrary $x,y \in C\cap B_{X}(0_{X},r)\setminus \cl{B}_{X}(0_{X},\delta)$ and observe that
		\begin{multline*}
			\norm{\Phi(y)-\Phi(x)}=\norm{\left(\frac{r-\frac{r\delta}{\norm{y}}}{r-\delta}\right)(y-x)+x\left(\frac{r-\frac{r\delta}{\norm{y}}}{r-\delta}-\frac{r-\frac{r\delta}{\norm{x}}}{r-\delta}\right)}\leq\\
			\left(\frac{r-\frac{r\delta}{\norm{y}}}{r-\delta}\right)\norm{y-x}+\frac{\frac{r\delta}{\norm{y}}\abs{\norm{y}-\norm{x}}}{(r-\delta)}\leq\left(1+\frac{\delta}{r-\delta}\right)\norm{y-x}.  
		\end{multline*}
		This completes the proof of \eqref{Phi3}. To prove \eqref{Phi4} we fix $x\in C$ arbitrarily and distinguish two cases: If $x\in C\cap \cl{B}_{X}(0_{X},\delta)$ or $x\in C\setminus B_{X}(0_{X},r)$ we clearly have $\norm{\Phi(x)-x}\leq \delta$. In the remaining case, $x\in C\cap B_{X}(0_{X},r)\setminus \cl{B}_{X}(0_{X},\delta)$, we get
		\begin{equation*}
			\norm{\Phi(x)-x}=\frac{\delta(r-\norm{x})}{r-\delta}\leq \delta.
		\end{equation*}
		This proves \eqref{Phi4} and completes the proof of the lemma.	
	\end{proof}	
	
	\begin{lemma}\label{lemma:tentpeg}
		Let $X$ be a normed space, $s\in (0,\infty)$ and $C\subseteq X$ be a convex set with $\diam C\geq s$. Then for every $z\in C$ there exists $e_{z}\in \Sph_{X}$ such that the line segment $\left[z,z+\frac{s}{3}e_{z}\right]$ is a subset of $C$. 
	\end{lemma}
	\begin{proof}
		Fix $v,w\in C$ so that $\norm{w-v}> 2s/3$. Then we define a mapping $C\to \Sph_{X}$, $z\mapsto e_{z}$ by
		\begin{equation*}
			e_{z}:=\begin{cases}
				\frac{v-z}{\norm{v-z}} & \text{ if }\norm{v-z}\geq \frac{s}{3},\\
				\frac{w-z}{\norm{w-z}} & \text{ otherwise.}
			\end{cases}
		\end{equation*}
		The mapping $z\mapsto e_{z}$ is well-defined and the line segment $\left[z,z+\frac{s}{3}e_{z}\right]$ is contained in $C$ (either in $[z,v]$ or in $[z,w]$) for each $z\in C$.
	\end{proof}

	\begin{lemma}\label{lemma:every_village}
		Let $X$ be a normed space, $C\subseteq X$ be a convex and bounded set containing $0_{X}$, $s\in(0,1)$, $\Gamma\subseteq C$ be an $s$-separated, non-empty, non-singleton~set, $f\colon C\to C$ be a non-expansive mapping and $\varepsilon\in(0,1)$. Then there exists a non-expansive mapping $g\colon C\to C$ 
		such that $\supnorm{g-f}\leq \varepsilon$ and
		\begin{equation}\label{eq:bump}
			\norm{g(y)-g(x)}=\norm{y-x}
		\end{equation}
		for all $x\in \Gamma$ and $\displaystyle y\in C\cap \cl{B}_{X}\left(x,\frac{\varepsilon s}{12(1+\diam C)}\right)$.
	\end{lemma}
	\begin{proof}
		To simplify the formulae in the proof we set $r=s/\hl{2}$. Let $\delta\in (0,r/2)$ be a parameter to be determined in the course of the proof and for each $x\in \Gamma$, let $\Phi_{x}\colon C-x\to C-x$ be the mapping $\Phi$ given by the conclusion of Lemma~\ref{lemma:flat} applied to $C-x$, $r$ and $\delta$. Here $C-x$ denotes the set $\set{y-x\colon y\in C}$. We define $g_{0}\colon C\to C$ by
		\begin{equation*}
			g_{0}(z)=\begin{cases}
				f(z) & \text{ if }z\in C\setminus \bigcup_{x\in\Gamma}B_{X}(x,r)\\
				f(x+\Phi_{x}(z-x)) & \text{ if }z\in C\cap B_{X}(x,r) \text{ and } x\in \Gamma.
			\end{cases}
		\end{equation*}
		Note that $g_{0}$ is a well-defined, continuous mapping $C\to C$,
		\begin{equation*}
			\supnorm{g_{0}-f}\leq \delta, \qquad \lip(g_{0})\leq1+\frac{\delta}{r-\delta}
		\end{equation*}
		and that $g_{0}$ is constant on each set $C\cap \cl{B}_{X}(x,\delta)$ with $x\in \Gamma$. Put differently, the latter condition is
		\begin{equation}\label{eq:g0_cnst}
			g_{0}(z)=g_{0}(x) \qquad\text{whenever }z\in C\cap\cl{B}_{X}(x,\delta)\text{ and }x\in\Gamma.
		\end{equation}
		Next, we define $g_{1}\colon C\to C$ by
		\begin{equation*}
			g_{1}:=\left(1-\frac{\delta}{r}\right)g_{0}.
		\end{equation*}
		Note that $g_{1}$ is a well-defined mapping $C\to C$, since $0_{X}\in C$, $C$ is convex and $g_{0}\colon C\to C$. Moreover, we have 
		\begin{equation*}
			\supnorm{g_{1}-g_{0}}\leq \frac{\delta\diam C}{r},\qquad \lip(g_{1})\leq 1
		\end{equation*}
		and that $g_{1}$ inherits property \eqref{eq:g0_cnst} from $g_{0}$: 
		\begin{equation}\label{eq:g1_cnst}
			g_{1}(z)=g_{1}(x) \qquad\text{whenever }z\in C\cap\cl{B}_{X}(x,\delta)\text{ and }x\in\Gamma.
		\end{equation}
		Let the family of directions $(e_{z})_{z\in C}\subseteq\Sph_{X}$ be given by the conclusion of Lemma~\ref{lemma:tentpeg}. For the family of directions $(u_{x})_{x\in \Gamma}\subseteq \Sph_{X}$, defined by
		\begin{equation*}
			u_{x}:=e_{g_{1}(x)},\qquad x\in \Gamma,
		\end{equation*}
		we infer that the line segment $\left[g_{1}(x),g_{1}(x)+\frac{s}{3}u_{x}\right]$ is a subset of $C$ for each $x\in\Gamma$.	 
		We may now define $g_{2}\colon C\to C$ by
		\begin{equation*}
			g_{2}(z)=\begin{cases}
				g_{1}(z) & \text{ if }z\in C\setminus \bigcup_{x\in\Gamma}B_{X}(x,\delta),\\
				\hl{g_{1}(x)+(\delta-\norm{z-x})u_{x} & \hl{\text{ if }z\in C\cap B_{X}(x,\delta)\setminus B_{X}(x,\delta/2) \text{ and }x\in \Gamma}} \\	
				\hl{g_{1}(x)+\norm{z-x}u_{x}} & \hl{\text{ if } z\in C\cap B_{X}(x,\delta/2) \text{ and }x\in\Gamma.}
				
			\end{cases}
		\end{equation*}
		Note that $g_{2}$ is a well-defined, continuous mapping $C\to C$, where the continuity relies on \eqref{eq:g1_cnst}. Moreover, we have
		\begin{equation*}
			\supnorm{g_{2}-g_{1}}\leq \delta/2\qquad\text{ and }\qquad\lip(g_{2})\leq 1.
		\end{equation*}
		Setting $g=g_{2}$ we obtain a non-expansive mapping $C\to C$ satisfying 
		\begin{multline*}
			\supnorm{g-f}\leq \supnorm{g_{2}-g_{1}}+\supnorm{g_{1}-g_{0}}+\supnorm{g_{0}-f}\\
			\leq \frac{\delta}{2}+\frac{\delta\diam C}{r}+\delta\leq \frac{3\delta(1+\diam C)}{r}.
		\end{multline*}
		Thus, we achieve $\supnorm{g-f}\leq \varepsilon$ and \eqref{eq:bump} by setting
		\begin{equation*}
			\delta:=\frac{\varepsilon r}{3(1+\diam C)}.
		\end{equation*}
	\end{proof}
	\section{Porosity.}
	This final section is devoted to the proofs of the main results Theorems~\ref{thm:main} and \ref{thm:dual}. 
	\subsection{Proving Theorem~\ref{thm:main}.}
	\begin{lemma}\label{lemma:porous}
		Let $X$ be a normed space, $C\subseteq X$ be a bounded, convex, non-empty set and $\mc{M}=\mc{M}(C)$ denote the space of non-expansive mappings $C\to C$ equipped with the supremum metric. Let $s\in (0,1)$, $\Gamma\subseteq C$ be an $s$-separated, non-empty, non-singleton set, $\lambda\in(0,1)$ and $\mc{N}_{\lambda,\Gamma,s}$ denote the set of mappings $f\in \mc{M}$ for which 
		\begin{equation*}
			\inf_{x\in \Gamma}\lip(f,x,s)\leq \lambda.
		\end{equation*}
		Then $\mc{N}_{\lambda,\Gamma,s}$ is a lower porous subset of $\mc{M}$. In fact $\mc{N}_{\lambda,\Gamma,s}$ is lower porous at every point of $\mc{M}$. 
	\end{lemma}
	\begin{proof}
		We verify the $(\varepsilon_{0},\beta)$-condition of Lemma~\ref{lemma:def_porous}\eqref{def_lower_por} for the set $\mc{N}_{\lambda,\Gamma,s}$. We may assume without loss of generality that $0_{X}\in C$. Let $f\in \mc{M}$, $\varepsilon_{0}:=1$ and $\varepsilon\in(0,\varepsilon_{0})$. 
		
		Let the mapping $g\in\mc{M}$ satisfying $\supnorm{g-f}\leq\varepsilon$ and \eqref{eq:bump} be given by the conclusion of Lemma~\ref{lemma:every_village} applied to $X$, $C$, $s$, $\Gamma$, $f$ and $\varepsilon$. Let $\beta=\beta(\lambda,\Gamma,s,C)\in (0,1)$ be a parameter to be determined later in the proof and let $h\in B_{\mc{M}}(g,\beta \varepsilon)$. Our task is to show that $h\notin \mc{N}_{\lambda,\Gamma,s}$.
		
		Consider the family of directions $(e_{z})_{z\in C}\subseteq \Sph_{X}$ given by Lemma~\ref{lemma:tentpeg} and for each $x\in \Gamma$ define a point $y_{x}\in X$ by
		\begin{equation*}
			y_{x}:=x+\frac{\varepsilon s}{24(1+\diam C)}e_{x}.
		\end{equation*}
		Then,
		\begin{equation*}
			y_{x}\in C\cap B_{X}\left(x,\frac{\varepsilon s}{12(1+\diam C)}\right)\subseteq C\cap B_{X}(x,s)
		\end{equation*}
		for every $x\in \Gamma$. Now, exploiting property \eqref{eq:bump} of $g$, we derive
		\begin{equation*}
			\norm{h(y_{x})-h(x)}\geq \norm{g(y_{x})-g(x)}-2\beta\varepsilon= \left(1-\frac{48\beta(1+\diam C)}{s}\right)\norm{y_{x}-x}
		\end{equation*}
		for every $x\in \Gamma$, which implies
		\begin{equation*}
			\inf_{x\in \Gamma}\lip(h,x,s)\geq 1-\frac{48\beta(1+\diam C)}{s}>\lambda,
		\end{equation*}
		and therefore $h\notin \mc{N}_{\lambda,\Gamma,s}$, when we set
		\begin{equation*}
			\beta :=\frac{(1-\lambda)s}{96(1+\diam C)}.
		\end{equation*}
	\end{proof}

	We are now ready to prove Theorem~\ref{thm:main}:
	\begin{proof}[Proof of Theorem~\ref{thm:main}]
		For each $j\in\N$, let the set $\Gamma_{j}$ be chosen as a maximal $2^{-j}\diam C$-separated subset of $C$. For $j,k\in\N$, let $s_{j,k}:=2^{-j-k}\min\set{1,\diam C}$ and note that each set $\Gamma_{j}$ is $s_{j,k}$-separated for every $k\in\N$. 
		The subset $\mc{N}$ of $\mc{M}$ is then defined as
		\begin{equation*}
			\mc{N}:=\bigcup_{\lambda\in \Q\cap (0,1)}\bigcup_{j\in\N}\bigcup_{k\in\N}\mc{N}_{\lambda,\Gamma_{j},s_{j,k}},
		\end{equation*}
		where the sets $\mc{N}_{\lambda,\Gamma_{j},s_{j,k}}$ are given by Lemma~\ref{lemma:porous}. Then $\mc{N}$ is a $\sigma$-lower porous subset of $\mc{M}$. 
		
		Let now $f\in\mc{M}\setminus \mc{N}$. In view of Lemma~\ref{lemma:Gdelta}, to prove that $R(f)$ is residual in $C$, we only need to show that it is dense in $C$. We will show that $\bigcup_{j\in\N}\Gamma_{j}\subseteq R(f)$. Since the former set is dense in $C$, this will establish the density of $R(f)$ in $C$.
		
		Let $j\in\N$ and $x\in\Gamma_{j}$. Since $f\notin \bigcup_{\lambda\in \Q\cap (0,1)}\bigcup_{k\in\N}\mc{N}_{\lambda,\Gamma_{j},s_{j,k}}$, we have that $\lip(f,x,s_{j,k})> \lambda$ for every $k\in \N$ and $\lambda\in \Q\cap(0,1)$. Hence, $\lip(f,x,s_{j,k})=1$ for every $k\in\N$, which implies $\lip(f,x)=1$ and $x\in R(f)$. 
	\end{proof}
	\subsection{Reich's Question}\label{subsec:Reich}
	Before moving on to the proof of Theorem~\ref{thm:dual}, we note that Lemma~\ref{lemma:porous} and its foundations in Section~\ref{sec:perturb} provide an elementary answer to Reich's question (Question~\ref{q:reich}). Theorem~\ref{thm:elementary} actually provides a strong affirmative answer to Question~\ref{q:reich}: observe that the condition $\lip(f,x)<1$ defining the set $\mc{V}_{x}$ in Theorem~\ref{thm:elementary} is weaker than the condition $\lip(f)<1$ defining the set of strict contractions. We also note that \cite[Theorem~2.2]{bargetz_dymond2016} is an immediate corollary of Theorem~\ref{thm:elementary} and Lemma~\ref{lemma:Gdelta}.
	\begin{thm}\label{thm:elementary}
		Let $X$ be a Banach space, $C\subseteq X$ be a bounded, convex, non-empty, non-singleton set, $x\in C$ and $\mc{M}=\mc{M}(C)$ denote the space of non-expansive mappings $C\to C$ equipped with the supremum metric. Then the set
		\begin{equation*}
			\mc{V}_{x}:=\set{f\in\mc{M}\colon \lip(f,x)<1}
		\end{equation*}
		is $\sigma$-lower porous in $\mc{M}$.
	\end{thm}
	\begin{proof}
		Fix $y\in C\setminus\set{x}$ and set $s=\min\set{\norm{y-x},1/2}$ and $\Gamma=\set{x,y}$. Then the set $\mc{V}_{x}$ is a subset of the countable union
		\begin{equation*}
			\bigcup_{\lambda\in\Q\cap(0,1)}\bigcup_{j\in\N}\mc{N}_{\lambda,\Gamma,s/j},
		\end{equation*}	
		where the lower porous subsets $\mc{N}_{\lambda,\Gamma,s\hl{/j}}$ of $\mc{M}$ are given by Lemma~\ref{lemma:porous}.
	\end{proof}
	
	\subsection{Proving Theorem~\ref{thm:dual}.}
	\begin{lemma}\label{lemma:phiinv}
		Let $K>1$ and $\phi,\hl{\xi}\colon (0,1/K)\to (0,\infty)$ be strictly increasing, concave functions satisfying 
		\begin{equation*}
			\phi(t)\hl{\xi}(t)\geq \frac{t}{K}\quad\text{ for all }t\in (0,1/K)\qquad\text{ and }\qquad\lim_{t\to 0}\phi(t)=\lim_{t\to 0}\hl{\xi}(t)=0.
		\end{equation*}
		Then the function 
		\begin{equation*}
			(0,\sup\phi)\to (0,\infty),\qquad t\mapsto \frac{\phi^{-1}(t)}{t}
		\end{equation*}
		is increasing and satisfies
		\begin{equation*}
			\lim_{t\to 0}\frac{\phi^{-1}(t)}{t}=0.
		\end{equation*}
	\end{lemma}
	\begin{proof}
		Note that $\phi^{-1}$ is a strictly increasing, convex function defined on the interval $(0,\sup\phi)$ and that $\lim_{t\to 0}\phi^{-1}(t)=0$. It follows that the function $\frac{\phi^{-1}(t)}{t}$ is increasing on $(0,\sup\phi)$. Given $\varepsilon>0$, we may choose $s>0$ sufficiently small so that $\hl{\xi}(s)<\varepsilon/K$. Then for every $t\in (0,\phi(s))$ we have
		\begin{equation*}
			\frac{\phi^{-1}(t)}{t}\leq \frac{s}{\phi(s)}\leq K\hl{\xi}(s)<\varepsilon.
		\end{equation*} 
	\end{proof}

	\begin{lemma}\label{lemma:porous2}
		Let $X$ be a normed space, $C\subseteq X$ be a bounded, convex~set and $\mc{M}=\mc{M}(C)$ denote the space of non-expansive mappings $C\to C$ equipped with the supremum metric. Let $\lambda\in (0,1)$, $k\in\N$, $K>1$ and $\phi,\hl{\xi}\colon(0,1/K)\to(0,\infty)$ be strictly increasing, concave functions satisfying 
		\begin{equation*}
			\phi(t)\hl{\xi}(t)\geq \frac{t}{K}\,\text{ for all }t\in (0,1/K)\quad\text{ and }\quad\lim_{t\to 0}\phi(t)=\lim_{t\to 0}\hl{\xi}(t)=0.
		\end{equation*}
		Let $\mathfrak{s}=(s_{j})_{j\in\N}$ be a strictly decreasing sequence of numbers in the interval $(0,1)\cap (0,\sup\phi)\cap (0,\diam C/2)$ converging to $0$ and satisfying 
		\begin{equation}\label{eq:sk}
			\frac{\phi^{-1}(s_{j})}{s_{j}}=\frac{2\phi^{-1}(s_{j+1})}{s_{j+1}}\qquad\text{ for all }j\in\N.
		\end{equation}
		Let $\mb{\Gamma}=(\Gamma_{j})_{j\in\N}$ be a sequence of subsets of $C$, where $\Gamma_{j}$ is $s_{j}$-separated, non-empty and non-singleton for each $j\in\N$. Let $\mc{P}_{\lambda,\mb{\Gamma},\mf{s},\phi,k}$ denote the set of mappings $f\in \mc{M}$ with the property that for every $j\geq k$ there exist $x\in \Gamma_{j}$ and $y\in C\cap B_{X}\left(x,\frac{(1-\lambda)\phi^{-1}(s_{j})}{48(1+\diam C)}\right)$ such that
		\begin{equation*}
			\lip\left(f,y,\phi^{-1}(s_{j})\right)\leq \lambda.
		\end{equation*}
		Then $\mc{P}_{\lambda,\mb{\Gamma},\mf{s},\phi,k}$ is a $\hl{\xi}$-lower porous subset of $\mc{M}$.
		In fact $\mc{P}_{\lambda,\mb{\Gamma},\mf{s},\phi,k}$ is $\hl{\xi}$-lower porous at every point of $\mc{M}$. 
	\end{lemma}
	\begin{proof}
		We verify the $(\varepsilon_{0},\beta)$-condition of Lemma~\ref{lemma:def_porous}\eqref{def_lower_por} for the set $\mc{P}_{\lambda,\mb{\Gamma},\mf{s},\phi,k}$. We may assume without loss of generality that $0_{X}\in C$.
		
		Let $f\in\mc{M}$, $\varepsilon_{0}=\min\set{\frac{\phi^{-1}(s_{k})}{s_{k}},1}$ and $\varepsilon\in (0,\varepsilon_{0})$. In view of Lemma~\ref{lemma:phiinv}, there is a unique integer $j\geq k$ satisfying
		\begin{equation}\label{eq:choice_j}
			\frac{\phi^{-1}(s_{j+1})}{s_{j+1}}<\varepsilon\leq \frac{\phi^{-1}(s_{j})}{s_{j}}.
		\end{equation}
		For future reference, we point out that the last inequality of \eqref{eq:choice_j} implies
		\begin{equation}\label{eq:equv_choice_j}
			\hl{\xi}^{-1}(\varepsilon/K)\leq \phi^{-1}(s_{j}).
		\end{equation}
		To verify this implication, apply the inequality $\phi(t)\hl{\xi}(t)\geq t/K$ in \eqref{eq:choice_j} to obtain $\varepsilon\leq K\hl{\xi}(\phi^{-1}(s_{j}))$. Finally, divide through by $K$ and apply the strictly increasing function $\hl{\xi}^{-1}$ to both sides.
		
		Let the mapping $g\in\mc{M}$ be given by the conclusion of Lemma~\ref{lemma:every_village} applied to $f$, $\varepsilon$, $\Gamma=\Gamma_{j}$ and $s=s_{j}$. Then we have $\supnorm{g-f}\leq\varepsilon$. Let $\beta=\beta(\lambda,\mb{\Gamma},\mf{s},\phi,K,k)\in (0,1)$ be a parameter to be determined later in the proof and let $h\in B_{\mc{M}}\left(g,\hl{\xi}^{-1}(\beta\varepsilon)\right)$ be arbitrary. To complete the proof, we show that $h\notin \mc{P}_{\lambda,\mb{\Gamma},\mf{s},\phi,k}$. 
		
		Let $x\in \Gamma_{j}$ and $y\in C\cap B_{X}\left(x,\frac{(1-\lambda)\phi^{-1}(s_{j})}{48(1+\diam C)}\right)$. Further, let the direction $e_{x}\in\Sph_{X}$ be given by Lemma~\ref{lemma:tentpeg} applied to $s=s_{j}$ and define 
		\begin{equation*}
			z:=x+\frac{\phi^{-1}(s_{j})}{24(1+\diam C)}e_{x}\in C.
		\end{equation*}
		Using \eqref{eq:sk} and the first inequality of \eqref{eq:choice_j}, we verify $y,z\in B_{X}\left(x,\frac{\varepsilon s_{j}}{12(1+\diam C)}\right)$, so that property \eqref{eq:bump} of $g$ applies to both $y$ and $z$. We moreover have 
		\begin{equation*}
			\frac{(1+\lambda)\phi^{-1}(s_{j})}{48(1+\diam C)}\leq \norm{z-y} \leq \frac{(3-\lambda)\phi^{-1}(s_{j})}{48(1+\diam C)}<\phi^{-1}(s_{j}),
		\end{equation*}
		which allows us to use the point $z$ to estimate $\lip\left(h,y,\phi^{-1}(s_{j})\right)$. Using also the property \eqref{eq:bump} of $g$, we derive
		\begin{align}\label{eq:inequalities}
			\lip\left(h,y,\phi^{-1}(s_{j})\right)\norm{z-y}&\geq\norm{h(z)-h(y)}\nonumber\\
			&\geq \norm{g(z)-g(y)}-2\hl{\xi}^{-1}(\beta\varepsilon)\nonumber\\
			&\geq \norm{g(z)-g(x)}-\norm{g(y)-g(x)}-2\hl{\xi}^{-1}(\beta\varepsilon)\nonumber\\
			&\geq\frac{\phi^{-1}(s_{j})}{24(1+\diam C)}-\frac{(1-\lambda)\phi^{-1}(s_{j})}{48(1+\diam C)}-2\hl{\xi}^{-1}(\beta\varepsilon)\nonumber\\
			&\geq \left(\frac{1+\lambda}{3-\lambda}-\frac{96\beta K(1+\diam C)\hl{\xi}^{-1}(\varepsilon/K)}{(1+\lambda)\phi^{-1}(s_{j})}\right)\norm{z-y}
			\nonumber\\
			&\geq \left(\frac{(1+\lambda)^{2}-96(3-\lambda)\beta K(1+\diam C)}{(1+\lambda)(3-\lambda)}\right)\norm{z-y}.
		\end{align}
		To get the penultimate inequality of \eqref{eq:inequalities} we use the inequality 
		\begin{equation*}
			\hl{\xi}^{-1}(\beta\varepsilon)\leq \beta K\hl{\xi}^{-1}(\varepsilon/K),
		\end{equation*}
		which holds, provided that
		\begin{equation}\label{eq:cond_beta}
			\beta K<1,
		\end{equation}
		because of the convexity of $\hl{\xi}^{-1}$ and $\lim_{t\to 0}\hl{\xi}^{-1}(t)=0$. The last inequality of \eqref{eq:inequalities} is due to \eqref{eq:equv_choice_j}. Finally, we prescribe that
		\begin{equation*}
			\beta:=\frac{(1-\lambda)^{2}(1+\lambda)}{97(3-\lambda)K(1+\diam C)},
		\end{equation*}
		so that \eqref{eq:cond_beta} is satisfied and \eqref{eq:inequalities} implies
		\begin{equation*}
			\lip\left(h,y,\phi^{-1}(s_{j})\right)>\lambda.
		\end{equation*}
		Since $x\in\Gamma_{j}$ and $y\in C\cap B_{X}\left(x,\frac{(1-\lambda)\phi^{-1}(s_{j})}{48(1+\diam C)}\right)$ were arbitrary and $j\geq k$, this argument establishes $h\notin \mc{P}_{\lambda,\mb{\Gamma},\mf{s},\phi,k}$. 
	\end{proof}
	The next lemma can be thought of as the dual of Lemma~\ref{lemma:porous2}.
	\begin{lemma}\label{lemma:porous3}
		Let $X$, $C$, $\mc{M}$, $\lambda$, $K$, $\phi$, $\hl{\xi}$, $\mf{s}=(s_{j})_{j\in\N}$ and $\mb{\Gamma}=(\Gamma_{j})_{j\in\N}$ satisfy the conditions of Lemma~\ref{lemma:porous2} and let the sets $\mc{P}_{\lambda,\mb{\Gamma},\mf{s},\phi,k}$ for $k\in\N$ be given by the conclusion of Lemma~\ref{lemma:porous2}. Assume, in addition, that each set $\Gamma_{j}$ is $s_{j}$-dense in $C$. Let $f\in\mc{M}\setminus \bigcup_{k\in\N}\mc{P}_{\lambda,\mb{\Gamma},\mf{s},\phi,k}$, $l\in\N$ and let 
		\begin{equation}\label{eq:setE}
			E_{\lambda,\mf{s},\phi,l}(f):=\set{x\in C\colon \sup_{j\geq l}\lip\left(f,x,\phi^{-1}(s_{j})\right)\leq \lambda}.
		\end{equation}
		Then $E_{\lambda,\mf{s},\phi,l}(f)$ is a $\phi$-upper porous subset of $C$. In fact $E_{\lambda,\mf{s},\phi,l}(f)$ is $\phi$-upper porous at every point of $C$.
	\end{lemma}
	\begin{proof}We verify the $(\alpha,\varepsilon)$-condition of Lemma~\ref{lemma:def_porous}\eqref{def_upper_por} for the set $E_{\lambda,\mf{s},\phi,l}(f)$.
		
		Let $z\in C$ and $\varepsilon>0$. Fix an integer $k\geq l$ large enough so that $s_{k}\leq\varepsilon$. Since $f\notin \mc{P}_{\lambda,\mb{\Gamma},\mf{s},\phi,k}$, there exists $j\geq k$ such that for all $x\in \Gamma_{j}$ and all $y\in C\cap B_{X}\left(x,\frac{(1-\lambda)\phi^{-1}(s_{j})}{48(1+\diam C)}\right)$ we have
		\begin{equation}\label{eq:f_notin_P}
			\lip\left(f,y,\phi^{-1}(s_{j})\right)>\lambda.
		\end{equation}
		Since $\Gamma_{j}$ is $s_{j}$-dense in $C$, we may choose $x\in\Gamma_{j}$ with $\norm{x-z}\leq s_{j}\leq \varepsilon$. Then, \eqref{eq:f_notin_P} holds for every $y\in C\cap B_{X}\left(x,\frac{(1-\lambda)\phi^{-1}(s_{j})}{48(1+\diam C)}\right)$, which implies $B_{X}\left(x,\frac{(1-\lambda)\phi^{-1}(s_{j})}{48(1+\diam C)}\right)\cap E_{\lambda,\mf{s},\phi,l}(f)=\emptyset$. Moreover, we have 
		\begin{equation*}
			\frac{(1-\lambda)\phi^{-1}(s_{j})}{48(1+\diam C)}\geq \frac{(1-\lambda)\phi^{-1}(\norm{x-z})}{48(1+\diam C)}\geq  \phi^{-1}\left(\frac{(1-\lambda)\norm{x-z}}{48(1+\diam C)}\right),
		\end{equation*}
		where the last inequality is due to the convexity of $\phi^{-1}$ and $\lim_{t\to 0}\phi^{-1}(t)=0$. Therefore, 
		\begin{equation*}
			B_{X}\left(x,\phi^{-1}(\alpha\norm{x-z})\right)\cap E_{\lambda,\mf{s},\phi,l}(f)=\emptyset,
		\end{equation*}
		with
		\begin{equation*}
			\alpha:=\frac{1-\lambda}{48(1+\diam C)}.
		\end{equation*}
	\end{proof}
	
	\hl{\begin{remark*}
			At this point the reader may ask whether the sets $E_{\lambda,\mf{s},\phi,l}(f)$ of Lemma~\ref{lemma:porous3} could even be shown to be $\phi$-lower porous. However, the current arguments do not appear sufficient to achieve this. Intuitively, the difference between upper and lower porous lies in whether the picture of porosity at a point (seeing close to the point a relatively large hole) occurs at all scales below a certain threshold or just along a specific sequence of scales converging to zero.	
			
			In the setting of Lemma~\ref{lemma:porous3}, a mapping $f\in \mc{M}$ not belonging to $\bigcup_{k\in\N}\mc{P}_{\lambda,\Gamma,\mf{s},\phi,k}$ means that there is a specific sequence $(j_{k})_{k\in\N}$ of numbers $j_{k}\geq k$ witnessing $f$'s failure of the `for every $j\geq k$ property' defining $\mc{P}_{\lambda,\Gamma,\mf{s},\phi,k}$ in Lemma~\ref{lemma:porous2}. In the proof of Lemma~\ref{lemma:porous3} it is shown that this sequence $(j_{k})_{k\in\N}$ corresponds to a specific sequence $(\eta_{k})_{k\in\N}$ of scales $\eta_{k}\searrow 0$ at which we see the porosity of the set $E_{\lambda,\mf{s},\phi,l}(f)$ at the point $z$. (In the proof of Lemma~\ref{lemma:porous3}, $\eta_{k}$ may be taken say as the quantity $3\dist(z,\Gamma_{j_{k}})$. In a picture of scale $\eta_{k}$ centered at the point $z$ we see a relatively large hole). Since the scales at which we see a porosity picture are tied to the sequence $(j_{k})_{k\in\N}$, which may grow arbitrarily fast, the arguments do not deliver lower porosity.
	\end{remark*}}

	\begin{proof}[Proof of Theorem~\ref{thm:dual}]
		If $\lim_{t\to 0}\phi(t)>0$ then $\phi$-upper porosity is exactly nowhere density and so the conclusion of Theorem~\ref{thm:dual} is implied by that of Theorem~\ref{thm:main}. Therefore, we may assume that $\lim_{t\to 0}\phi(t)=0$.	
		
		Define a sequence $\mf{s}=(s_{j})_{j\in\N}$ of positive numbers $s_{j}$ inductively as follows: Begin by setting $s_{1}:=\frac{1}{4}\min\set{1,\sup\phi,\diam C}$. If $j\in\N$ and $s_{j}\in (0,1)\cap (0,\sup\phi)\cap (0,\diam C/2)$ is already defined, we choose $s_{j+1}\in (0,s_{j})$ such that
		\begin{equation*}
			\frac{\phi^{-1}(s_{j+1})}{s_{j+1}}=\frac{\phi^{-1}(s_{j})}{2s_{j}}.
		\end{equation*}
		Such a choice exists due to Lemma~\ref{lemma:phiinv}. The sequence $(s_{j})_{j\in\N}$ is now defined and is strictly decreasing. Since $\frac{\phi^{-1}(s_{j})}{s_{j}}\leq 2^{-j+1}\frac{\phi^{-1}(s_{1})}{s_{1}}\to 0$ as $j\to \infty$, the properties of the function $t\mapsto \frac{\phi^{-1}(t)}{t}$ established in Lemma~\ref{lemma:phiinv} ensure that $\lim_{j\to\infty}s_{j}=0$.
		
		For each $j\in\N$ let $\Gamma_{j}$ be a maximal $s_{j}$-separated subset of $C$ and set $\mb{\Gamma}:=(\Gamma_{j})_{j\in\N}$. For each $\lambda\in (0,1)$ and every $k,l\in\N$, the objects $X$, $C$, $\mc{M}$, $\lambda$, $K$, $\phi$, $\hl{\xi}$, $\mf{s}$, $\mb{\Gamma}$, $k$ and $l$ now satisfy the conditions of Lemmas~\ref{lemma:porous2}~and~\ref{lemma:porous3}. For each $\lambda\in (0,1)$ and $k\in \N$ let the $\hl{\xi}$-lower porous subset $\mc{P}_{\lambda,\mb{\Gamma},\mf{s},\phi,k}$ of $\mc{M}$ be given by Lemma~\ref{lemma:porous2}. For each $\lambda\in (0,1)$, $l\in\N$ and $f\in\mc{M}$, let the subset $E_{\lambda,\mf{s},\phi,l}(f)$ of $C$ be given by \eqref{eq:setE}. Then it is readily verified that
		\begin{equation*}
			C\setminus R(f)\subseteq \bigcup_{\lambda\in \Q\cap (0,1)}\bigcup_{l\in\N}E_{\lambda,\mf{s},\phi,l}(f)
		\end{equation*}
		for every $f\in \mc{M}$. Moreover, by Lemma~\ref{lemma:porous3} the latter union of sets is $\sigma$-$\phi$-upper porous whenever $f\in\mc{M}\setminus\mc{N}$, where $\mc{N}$ is the $\sigma$-$\hl{\xi}$-lower porous subset of $\mc{M}$ given by
		\begin{equation*}
			\mc{N}:=\bigcup_{\lambda\in \Q\cap (0,1)}\bigcup_{k\in\N}\mc{P}_{\lambda,\mb{\Gamma},\mf{s},\phi,k}.
		\end{equation*}
		
	\end{proof}
	\bibliographystyle{plain}
	\bibliography{citations}
	
	\noindent Michael Dymond\\Mathematisches Insitut\\
	Universität Leipzig\\
	PF 10 09 02\\
	04109 Leipzig\\
	Germany\\
	\texttt{michael.dymond@math.uni-leipzig.de}\\[3mm]
\end{document}